\documentclass[12pt]{article}

\usepackage{amssymb}
\usepackage{amsthm}
\usepackage{graphicx,float}
\usepackage{epstopdf}
\usepackage{amsfonts}
\usepackage[numbers]{natbib}
\usepackage[colorlinks,citecolor=blue,urlcolor=blue]{hyperref}
\usepackage{amsmath}
\usepackage{mathabx}
\usepackage{geometry}
\newtheorem{theorem}{Theorem}[section]
\newtheorem{lemma}[theorem]{Lemma}
\newtheorem{proposition}[theorem]{Proposition}

\newtheorem{remark}[theorem]{Remark}

\newcommand{\dd}{\ensuremath{\operatorname{d}\! }}
\newcommand{\dt}{\ensuremath{\operatorname{d}\! t}}
\newcommand{\ds}{\ensuremath{\operatorname{d}\! s}}

\newcommand{\dw}{\ensuremath{\operatorname{d}\! W}}

\begin{document}

\title{Infinite horizon discounted LQ optimal control problems\\
for mean-field switching diffusions\thanks{This work was supported
by the National Natural Science Foundation of China (12471414),
the Natural Science Foundation of Jiangsu Province (BK20242023),
and the Jiangsu Province Scientific Research Center of Applied Mathematics (BK20233002).}}

\author{
Kai Ding\thanks{School of Mathematics, Southeast University,
Nanjing 211189, China (dkaiye@163.com).}
\and
Xun Li\thanks{Department of Applied Mathematics,
The Hong Kong Polytechnic University, Kowloon, Hong Kong, China
(li.xun@polyu.edu.hk).}
\and
Siyu Lv\thanks{School of Mathematics, Southeast University,
Nanjing 211189, China (lvsiyu@seu.edu.cn).}
\and
Zuo Quan Xu\thanks{Department of Applied Mathematics,
The Hong Kong Polytechnic University, Kowloon, Hong Kong, China
(maxu@polyu.edu.hk).}
}

\date{}

\maketitle

\begin{abstract}
This paper investigates an infinite horizon discounted linear-quadratic (LQ)
optimal control problem for stochastic differential equations (SDEs) incorporating
regime switching and mean-field interactions. The regime switching is modeled by
a finite-state Markov chain acting as common noise, while the mean-field interactions
are characterized by the conditional expectation of the state process given the history
of the Markov chain. To address system stability in the infinite horizon setting,
a discounted factor is introduced. Within this framework, the well-posedness of
the state equation and adjoint equation -- formulated as infinite horizon mean-field
forward and backward SDEs with Markov chains, respectively -- is established, along
with the asymptotic behavior of their solutions as time approaches infinity.
A candidate optimal feedback control law is formally derived based on two algebraic
Riccati equations (AREs), which are introduced for the first time in this context.
The solvability of these AREs is proven through an approximation scheme involving
a sequence of Lyapunov equations, and the optimality of the proposed feedback control
law is rigorously verified using the completion of squares method. Finally, numerical
experiments are conducted to validate the theoretical findings, including solutions
to the AREs, the optimal control process, and the corresponding optimal (conditional)
state trajectory. This work provides a comprehensive framework for solving infinite
horizon discounted LQ optimal control problems in the presence of regime switching
and mean-field interactions, offering both theoretical insights and practical
computational tools.
\end{abstract}



\section{Introduction}\label{section Introduction}

As one of the fundamental classes of optimal control theory, linear-quadratic (LQ for short) control problem
has been investigated for decades since the pioneering works by Kalman \cite{Kalman1960} (for deterministic
situation) and Wonham \cite{Wonham1968} and Bismut \cite{Bismut1976} (for stochastic situation). In recent
years, the stochastic LQ control theory experienced some new significant developments; see, for example,
indefinite stochastic LQ control (Chen et al. \cite{CLZ1998}), stochastic LQ control with random
coefficients (Tang \cite{Tang2003}), and the concepts of open-loop and closed-loop solutions
(Sun et al. \cite{SLY2016}), to mention a few.

It is well-known that optimal control problems can be divided, roughly speaking, into two categories:
one is the \emph{finite horizon} case and the other one is the \emph{infinite horizon} case.
In existing literature, the research of stochastic LQ control is concerned mainly with the finite
horizon case, but much less on the infinite horizon case. One possible reason is that one needs
\emph{additionally} to deal with the issue of system stability in an infinite horizon, which makes
the problem become more difficult and intractable; see Ait Rami and Zhou \cite{AZ2000}.
Moreover, the adjoint equation (in a stochastic maximum principle) for the infinite horizon case is
an infinite horizon backward stochastic differential equation (BSDE for short), whose well-posedness
is no longer evident as there is \emph{no} terminal value will be specified in advance;
see Wei and Yu \cite{WeiYu2021} for some latest results on infinite horizon stochastic differential
equations (SDEs for short) and BSDEs in a Brownian setting.

In many infinite horizon optimal control problems arising from various different fields of finance,
management, and engineering, a discount factor is commonly introduced in the cost (or reward) functional;
see \cite{HS2001,ZhangQing2001,CLP2013}. It not only avoids the subtle issue of system stability
and makes the involved infinite integrals be convergent so that one can focus on the optimal control
theory, but also has the realistic meaning to reflect the value of time in the long run. However,
such a practice seems, somewhat unexpectedly, unusual in the study of infinite horizon LQ control problems.

On the other hand, in the systems and control community there are now two topics attracted considerable
interests, one is the mean-field control, which takes the influence of \emph{collective behavior}
on participating individuals into consideration (see Buckdahn et al. \cite{BDLP2009} and Carmona
and Delarue \cite{CD2018}), whereas the other one is the regime switching diffusion to describe
the \emph{hybrid} phenomenon that continuous dynamics and discrete events coexist in many applications
(see Yin and Zhu \cite{YinZhu2010} and Yin and Zhang \cite{YinZhang2013}). It is not an easy task to
blend the two important topics together. Recently, Nguyen et al. \cite{NYH2020} established the theory
of (finite horizon) mean-field SDEs with regime switching, where the mean-field term is characterized
as the \emph{conditional expectation} of state process given the history of the Markov chain. This work
initiated the research of (finite horizon) mean-field optimal control problems with regime switching;
for example, Nguyen et al. \cite{NNY2020,NYN2021} obtained stochastic maximum principles in the local
and global forms, respectively, Lv et al. \cite{LXZ2023,LWX2024} investigated LQ leader-follower and
nonzero-sum stochastic differential games, respectively, Zhang et al. \cite{ZZM2024,ZZM2025}
studied mixed stochastic $H_{2} / H_{\infty}$ control problems via a Stackelberg game approach,
and Jian et al. \cite{JLSY2024} discussed the convergence rate of $N$-player linear quadratic Gaussian
game with a Markov chain common noise towards its asymptotic mean-field game.

In this paper, we continue to develop such a newly proposed theory in the direction of infinite horizon
case and discounted LQ framework. The motivation of this paper is twofold. First, there is relatively rare
existing literature dealing with \emph{infinite horizon} LQ control problems, which are crucial for describing
the long-term behavior of various real-world dynamic systems that cannot be easily incorporated by the finite
horizon framework. Second, despite the broad applications of regime switching modeling and mean-field method,
the analysis of combination of these two interesting research directions in infinite horizon framework is still
blank and remains to be treated.

The main contributions of this paper include the following three aspects:
\begin{itemize}
\item As a preliminary, we first establish the \emph{global well-posedness} of the state equation and
adjoint equation (as infinite horizon linear mean-field SDE and BSDE with Markov chains, respectively)
and the \emph{asymptotic property} of their solutions when time goes to infinity; here, a key point
is the selection of an appropriate discount factor (see Assumption (H1)).
\item Next, we adopt the so-called \emph{four-step scheme} to decouple the associated stochastic
Hamiltonian system consisting of the state equation and adjoint equation, and to derive two
algebraic Riccati equations (AREs for short). The solvability of these two AREs, which is interesting
in its own right, is obtained via an approximation approach of a sequence of \emph{Lyapunov equations}.
Then, we can formally construct a (candidate) optimal feedback control law based on solutions to the
two AREs; the optimality will be verified by the completion of squares method.
\item Finally, an example with a four-state Markov chain is provided to show the effectiveness of the
obtained results. Using the quasi-linearization method via Lyapunov equations, we numerically solve
the corresponding AREs, and display the behavior of optimal control process and optimal (conditional)
state trajectory.
\end{itemize}

The rest of this paper is organized as follows. Section \ref{section Problem formulation}
presents the problem formulation. Section \ref{Section Equation} establishes the solvability
of some infinite horizon mean-field SDEs and BSDEs with Markov chains. Section \ref{Section Riccati}
gives a formal derivation for the AREs and a candidate optimal feedback control.
Section \ref{Section Riccati solv} is devoted to the solvability of the AREs.
Then, Section \ref{Section optimality} verifies the optimality of the candidate.
Section \ref{Section Numerical experiments} reports a numerical experiment to illustrate
the theoretical results. Finally, some further concluding remarks are made
in Section \ref{Section Concluding remarks}.

\section{Problem formulation}\label{section Problem formulation}

Let $\mathbb{R}^{n}$ be the $n$-dimensional Euclidean space with Euclidean norm $|\cdot|$
and Euclidean inner product $\langle\cdot,\cdot\rangle$. Let $\mathbb{R}^{n\times m}$ be
the space consisting of all ($n\times m$) matrices. Let $\mathbb{S}^{n}\subset\mathbb{R}^{n\times n}$
be the set of all ($n\times n$) symmetric matrices, $\mathbb{S}^{n}_{+}$ be the set of
all ($n\times n$) positive semi-definite matrices $A$ (also denoted as $A\geq0$), and
$\widehat{\mathbb{S}}^{n}_{+}$ be the set of all ($n\times n$) positive definite matrices
$A$ (also denoted as $A>0$). Let $A^{\top}$ denote the transpose of a vector or matrix $A$,
let $A^{-1}$ denote the inverse of a square matrix $A$, let $1_{S}$ denote the indicator function
of a set $S$, and let $\mathbf{I}_n$ denote the $(n\times n)$ identity matrix.

Let $[0,\infty)$ be the infinite time horizon and $(\Omega,\mathcal{F},P)$ be a fixed probability
space on which a 1-dimensional standard Brownian motion $\{W(t),~t\geq0\}$, and a Markov chain
$\{\alpha(t),~t\geq0\}$, are defined. Assume conventionally that $W(\cdot)$ and $\alpha(\cdot)$
are independent. The Markov chain $\alpha$ takes values in a finite state space $\mathcal{M}=\{1,\ldots,M\}$.
Let $Q=(\lambda_{ij})_{i,j\in\mathcal{M}}$ be the generator (i.e., the matrix of transition rates)
of $\alpha$ with $\lambda_{ij}\geq 0$ for $i\neq j$ and $\sum_{j\in \mathcal{M}}\lambda_{ij}=0$
for each $i\in \mathcal{M}$.
For $t\geq0$, we denote $\mathcal{F}^{\alpha}_{t}=\sigma\{\alpha(s):0\leq s\leq t\}$,
$\mathcal{F}^{W}_{t}=\sigma\{W(s):0\leq s\leq t\}$,
and $\mathcal{F}_{t}\doteq\mathcal{F}^{W,\alpha}_{t}=\sigma\{W(s),\alpha(s):0\leq s\leq t\}$.

For each $i,j\in\mathcal{M}$ with $i\ne j$, let $N_{ij}(t)$ be the number of jumps of the Markov
chain $\alpha$ from state $i$ to state $j$ up to time $t$, i.e.,
$$
N_{ij}(t)=\sum_{0<s\leq t}1_{\{\alpha(s-)=i\}}1_{\{\alpha(s)=j\}}.
$$
It follows from Nguyen et al. \cite{NNY2020,NYN2021} that the process $M_{ij}(t)$ defined by
$$
M_{ij}(t)=N_{ij}(t)-\int_0^t \lambda_{ij}1_{\{\alpha(s-)=i\}}\ds
$$
is a purely discontinuous square integrable martingale with respect to $\mathcal{F}_{t}^{\alpha}$.
For any process $\Lambda(\cdot)=\{\Lambda_{ij}(\cdot)\}_{i,j\in\mathcal{M}}$, we denote for simplicity that
$$
\int_0^t \Lambda(s)\bullet \dd M(s)=\sum_{i,j\in\mathcal{M}}\int_0^t \Lambda_{ij}(s)\dd M_{ij}(s),
\quad\int_0^t \Lambda(s)\bullet \dd\;[M](s)=\sum_{i,j\in\mathcal{M}}\int_0^t \Lambda_{ij}(s)\dd\;[M_{ij}](s),
$$
where $[M_{ij}](t)$ denotes the optional quadratic variation process of $M_{ij}(t)$.

Let $r>0$ be a constant discount factor.
For a Euclidean space $\mathbb{H}$, let $L^{2,r}_{\mathcal{F}}(0,\infty;\mathbb{H})$
denote the set of all $\mathbb{H}$-valued $\mathcal{F}_{t}$-adapted processes
$\phi(\cdot)$ such that
$$
E\bigg[\int_{0}^{\infty}e^{-rt}|\phi(t)|^{2}\dt\bigg]<\infty;
$$
and let $M_{\mathcal{F}}^{2,r}(0,\infty;\mathbb{H})$ denote the set of all
$\Lambda(\cdot)=\{\Lambda_{ij}(\cdot)\}_{i,j\in\mathcal{M}}$ with each $\Lambda_{ij}(\cdot)$
being an $\mathbb{H}$-valued $\mathcal{F}_{t}$-adapted process such that
$$
E\bigg[\int_0^{\infty}e^{-rt}|\Lambda(t)|^2\bullet \dd\;[M](t)\bigg]
=\sum_{i,j\in\mathcal{M}}E\bigg[\int_0^{\infty}e^{-rt}|\Lambda_{ij}(t)|^2\dd\;[M_{ij}](t)\bigg]<\infty.
$$

In this paper, we consider the following controlled linear state equation on $[0,\infty)$:
\begin{equation}\label{state equation}
\left\{
\begin{aligned}
\dd X(t)=\;&\Big[A(\alpha(t))X(t)+\widehat{A}(\alpha(t))E[X(t)|\mathcal{F}_{t}^{\alpha}]+B(\alpha(t))u(t)\Big ]\dt\\
&+\Big[C(\alpha(t))X(t)+\widehat{C}(\alpha(t))E[X(t)|\mathcal{F}_{t}^{\alpha}]+D(\alpha(t))u(t)\Big]\dw(t),\\
X(0)=\;&x,\quad \alpha(0)=i,
\end{aligned}
\right.
\end{equation}
where $X(\cdot)$ is the state process valued in $\mathbb{R}^{n}$, $u(\cdot)$ is the control process
valued in $\mathbb{R}^{k}$, and $A(i)$, $\widehat{A}(i)$, $B(i)$, $C(i)$, $\widehat{C}(i)$, $D(i)$, $i\in\mathcal{M}$,
are constant matrices of suitable dimensions. It follows from Proposition \ref{prop:SDE-wellpose} that,
for any $u(\cdot)\in L^{2,r}_{\mathcal{F}}(0,\infty;\mathbb{R}^{k})$, the state equation (\ref{state equation})
admits a unique solution $X(\cdot)\in L^{2,r}_{\mathcal{F}}(0,\infty;\mathbb{R}^{n})$.
Each $u(\cdot)\in L^{2,r}_{\mathcal{F}}(0,\infty;\mathbb{R}^{k})$ is called an \emph{admissible control}.

\begin{remark}
In fact, the SDE \eqref{state equation} is obtained as the mean-square limit as $N\to\infty$ of
a system of interacting particles $X^{l,N}(\cdot)$ satisfying (see Nguyen et al. \cite{NYH2020}):
\[
\left\{
\begin{aligned}
\dd X^{l,N}(t)=\;&\Big[A(\alpha(t))X(t)+\widehat{A}(\alpha(t))\frac{1}{N}\sum_{l=1}^N X^{l,N}(t)+B(\alpha(t))u(t)\Big]\dt\\
&+\Big[C(\alpha(t))X(t)+\widehat{C}(\alpha(t))\frac{1}{N}\sum_{l=1}^N X^{l,N}(t)+D(\alpha(t))u(t)\Big]\dw^l(t),\\
X^{l,N}(0)=\;&x,\quad\alpha(0)=i,\quad 1\le l\le N.
\end{aligned}
\right.
\]
Here, $\{W^{l}(\cdot)\}_{l=1}^{N}$ represents a collection of independent standard Brownian motions
and the Markov chain $\alpha(\cdot)$ serves as a \emph{common noise} for all particles. Intuitively,
since all the particles depend on the history of $\alpha(\cdot)$, their average and thereby its limit
as $N\rightarrow\infty$ should also depend on $\alpha(\cdot)$. This intuition has been rigorously justified
by the law of large numbers recently established by Nguyen et al. \cite[Theorem 2.1]{NYH2020}, which
shows that the joint process $(\frac{1}{N}\sum_{l=1}^{N}X^{l,N}(\cdot),\alpha(\cdot))$ converges weakly to
$(\mu_{\alpha}(\cdot),\alpha(\cdot))$, where $\mu_{\alpha}(t)\doteq E[X(t)|\mathcal{F}_{t}^{\alpha}]$
and $X(\cdot)$ is exactly the solution to \eqref{state equation}. Since then, the study of such kind
of models and their applications have become an important topic in the area of mean-field control theory;
see, for example, \cite{NNY2020,NYN2021,GNY2024,JLSY2024}.
\end{remark}

In this paper, we want to minimize the following quadratic cost functional over all admissible
controls $u(\cdot)\in L^{2,r}_{\mathcal{F}}(0,\infty;\mathbb{R}^{k})$:
\begin{equation}\label{cost functional}
\begin{aligned}
J(x,i;u(\cdot))
=\;&\frac{1}{2}E\bigg[\int_{0}^{\infty}e^{-rt}\bigg(\Big\langle Q(\alpha(t))X(t),X(t)\Big\rangle\\
&\qquad+\Big\langle \widehat{Q}(\alpha(t))E[X(t)|\mathcal{F}_{t}^{\alpha}],E[X(t)|\mathcal{F}_{t}^{\alpha}]\Big\rangle
+\Big\langle R(\alpha(t))u(t),u(t)\Big\rangle\bigg)\dt\bigg],
\end{aligned}
\end{equation}
where $Q(i)$, $\widehat{Q}(i)$, $R(i)$, $i\in\mathcal{M}$, are constant symmetric matrices of suitable dimensions.

Throughout this paper, we impose the following assumptions:
\begin{equation}\label{H1}\tag{H1}
\left\{
\begin{aligned}
&r\mathbf{I}_n-[A(i)+A^{\top}(i)+C^{\top}(i)C(i)]>0,\quad i\in\mathcal{M},\\
&r\mathbf{I}_n-[A(i)+A^{\top}(i)+\widehat{A}(i)+\widehat{A}^{\top}(i)]>0,\quad i\in\mathcal{M},
\end{aligned}
\right.
\end{equation}
and
\begin{equation}\label{H2}\tag{H2}
Q(i)\ge0,\quad\widehat{Q}(i)\ge0,\quad R(i)>0,\quad i\in\mathcal{M}.
\end{equation}
Indeed, Assumption \eqref{H1} is a condition that gives the \emph{range} of the discount factor
$r$ and thereby guarantees the system's stability; whereas Assumption \eqref{H2} is the so-called
\emph{standard condition} in stochastic LQ control, which ensures the optimal control belongs to
$L^{2,r}_{\mathcal{F}}(0,\infty;\mathbb{R}^{k})$ immediately.
\begin{remark}
The positive definiteness conditions in \eqref{H1}-\eqref{H2} also imply uniform positive definiteness
as the state space $\mathcal{M}$ of the Markov chain is a finite set. As a consequence, the control
problem admits at most one optimal control.
\end{remark}
\begin{remark}
We emphasize that it is possible to consider more general forms of the state equation \eqref{state equation}
and of the cost functional \eqref{cost functional}, for example, involving the conditional expectation of the control
variable $E[u(t)|\mathcal{F}_{t}^{\alpha}]$. In this paper, we take the above forms in order to better illustrate
our main idea and simplify the presentation; see also \cite{NNY2020,NYN2021,GNY2024}.
\end{remark}

\begin{remark}
For simplicity of presentation, we consider in this paper only the homogeneous case of the problem. In fact,
there is no essential difficulty by incorporating non-homogeneous terms into the state equation \eqref{state equation}
and cost functional \eqref{cost functional}. On the other hand, the positive definite condition \eqref{H2}
could be weakened to the so-called uniform convexity condition (see, for example, Sun et al. \cite{SLY2016}),
which will be studied in our future work.
\end{remark}

\section{Conditional mean-field SDEs and BSDEs}\label{Section Equation}

In this section, we establish the global solvability of the state equation and the adjoint equation
(as infinite horizon mean-field SDE and BSDE with Markov chains, respectively) and the asymptotic
behavior of their solutions in the infinity.

In fact, from Nguyen et al. \cite[Lemma 2.3]{NYN2021}, for any fixed $T>0$ and any admissible control,
the state equation (\ref{state equation}) admits a unique solution $X(\cdot)$ on the finite horizon $[0,T]$ such that
$$
E\bigg[\int_{0}^{T}|X(t)|^2\dt\bigg]<\infty.
$$
However, it cannot extend automatically to the infinite horizon case (i.e., $T=\infty$) and ensure that
$$
E\bigg[\int_0^{\infty}|X(t)|^2\dt\bigg]<\infty.
$$
In the following, Proposition \ref{prop:SDE-wellpose} gives the global solvability of SDE (\ref{state equation})
on the infinite horizon and in a discounted solution space $L_{\mathcal{F}}^{2,r}(0,\infty;\mathbb{R}^{n})$.
First, we introduce the decomposition:
\begin{equation}\label{decomposition}
X(t)=(X(t)-E[X(t)|\mathcal{F}_{t}^{\alpha}])
+E[X(t)|\mathcal{F}_{t}^{\alpha}]\doteq\widecheck{X}(t)+\widehat{X}(t).
\end{equation}
Accordingly, for any $u(\cdot)\in L_{\mathcal{F}}^{2,r}(0,\infty;\mathbb{R}^{k})$, we denote
\begin{equation*}
u(t)=(u(t)-E[u(t)|\mathcal{F}_{t}^{\alpha}])
+E[u(t)|\mathcal{F}_{t}^{\alpha}]\doteq\widecheck{u}(t)+\widehat{u}(t),
\end{equation*}
and it follows from Jensen's inequality and Fubini's Theorem that
\begin{equation}\label{uhat-wellpose}
E\bigg[\int_0^{\infty} e^{-rt}|\widehat{u}(t)|^2\dt\bigg]
=E\bigg[\int_0^{\infty} e^{-rt}|E[u(t)|\mathcal{F}_{t}^{\alpha}]|^2\dt\bigg]
\le E\bigg[\int_0^{\infty} e^{-rt}|u(t)|^2\dt\bigg]<\infty.
\end{equation}
Hence, $\widehat{u}(\cdot)\in L_{\mathcal{F}}^{2,r}(0,\infty;\mathbb{R}^k)$ and then
$\widecheck{u}(\cdot)=u(\cdot)-\widehat{u}(\cdot)\in L_{\mathcal{F}}^{2,r}(0,\infty;\mathbb{R}^k)$.
It turns out that the processes $\widecheck{X}(\cdot)$ and $\widehat{X}(\cdot)$ satisfy
\begin{equation*}
\left\{
\begin{aligned}
\dd\widehat{X}(t)=\;&\Big[\Big(A(\alpha(t))+\widehat{A}(\alpha(t))\Big)\widehat{X}(t)
+B(\alpha(t))\widehat{u}(t)\Big ]\dt,\\
\widehat{X}(0)=\;&x,\quad \alpha(0)=i,
\end{aligned}
\right.
\end{equation*}
and
\begin{equation*}
\left\{
\begin{aligned}
\dd \widecheck{X}(t)=\;&\Big[A(\alpha(t))\widecheck{X}(t)+B(\alpha(t))\widecheck{u}(t)\Big ]\dt\\
&+\Big[C(\alpha(t))\widecheck{X}(t)+\Big(C(\alpha(t))+\widehat{C}(\alpha(t))\Big)\widehat{X}(t)
+D(\alpha(t))u(t)\Big]\dw(t),\\
\widecheck{X}(0)=\;&0,\quad \alpha(0)=i,
\end{aligned}
\right.
\end{equation*}
respectively. Then, we have the following solvability result for the state equation (\ref{state equation}).
\begin{proposition}\label{prop:SDE-wellpose}
Let Assumption \eqref{H1} hold. For any $u(\cdot)\in L_{\mathcal{F}}^{2,r}(0,\infty;\mathbb{R}^k)$,
the state equation \eqref{state equation} admits a unique solution
$X(\cdot)\in L_{\mathcal{F}}^{2,r}(0,\infty;\mathbb{R}^n)$. Furthermore, there exists a constant $K>0$,
which depends only on $r$ and coefficients $A$, $\widehat{A}$, $B$, $C$, such that
\begin{equation}\label{X-wellpose}
E\bigg[\int_0^\infty e^{-rt}|X(t)|^2\dt\bigg]
\leq K\bigg(|x|^2+E\bigg[\int_0^{\infty}e^{-rt}|u(t)|^2\dt\bigg]\bigg),
\end{equation}
and
\begin{equation}\label{Xlim}
\lim_{T\rightarrow \infty}E\Big[e^{-rT}|X(T)|^2\Big]=0.
\end{equation}
\end{proposition}
\begin{proof}
In the following, we treat $\widehat{X}(\cdot)$ and $\widecheck{X}(\cdot)$
separately according to the decomposition \eqref{decomposition}.
Let $\widetilde{A}(i)\doteq A(i)+\widehat{A}(i)$. Applying It\^{o}'s formula
to $e^{-rt}|\widehat{X}(t)|^2$ over $[0,T]$ for any $T>0$, we have
\begin{equation}\label{X2-Ito}
\begin{aligned}
E\Big[e^{-rT}|\widehat{X}(T)|^2\Big]
=\;&|x|^2+E\bigg[\int_0^T e^{-rt}\Big\langle\Big(\widetilde{A}(\alpha(t))+\widetilde{A}^{\top}(\alpha(t))
-r\mathbf{I}_n\Big)\widehat{X}(t),\widehat{X}(t)\Big\rangle\dt\bigg]\\
&+2E\bigg[\int_0^T e^{-rt}\Big\langle\widehat{X}(t),B(\alpha(t))\widehat{u}(t)\Big\rangle\dt\bigg].
\end{aligned}
\end{equation}
It follows from Assumption \eqref{H1} that $\widetilde{A}(i)+\widetilde{A}^\top(i)-r\mathbf{I}_n<-m\mathbf{I}_n$
for some $m>0$.  Hence
\begin{equation}\label{X2-Ito-esti}
E\Big[e^{-rT}|\widehat{X}(T)|^2\Big]
+\frac{m}{2}E\bigg[\int_0^T e^{-rt}|\widehat{X}(t)|^2\dt\bigg]
\le |x|^2+\frac{2}{m}E\bigg[\int_0^T e^{-rt}|B(\alpha(t))\widehat{u}(t)|^2\dt\bigg].
\end{equation}
Let $\sigma(t)\doteq[C(\alpha(t))+\widehat{C}(\alpha(t))]\widehat{X}(t)
+D(\alpha(t))u(t)\in L_{\mathcal{F}}^{2,r}(0,\infty;\mathbb{R}^n)$.
Applying It\^{o}'s formula to $e^{-rt}|\widecheck{X}(t)|^2$ over $[0,T]$ for any finite $T>0$, we obtain
\begin{equation*}\label{X1-Ito}
\begin{aligned}
&E\Big[e^{-rT}|\widecheck{X}(T)|^2\Big]\\
=\;&E\bigg[\int_0^T e^{-rt}\Big\langle\Big(A(\alpha(t))+A^{\top}(\alpha(t))+C^{\top}(\alpha(t))C(\alpha(t))
-r\mathbf{I}_n\Big)\widecheck{X}(t),\widecheck{X}(t)\Big\rangle\dt\bigg]\\
&+E\bigg[\int_0^T e^{-rt}\Big(2\Big\langle\widecheck{X}(t),B(\alpha(t))\widecheck{u}(t)
+C^{\top}(\alpha(t))\sigma(t)\Big\rangle+|\sigma(t)|^2\Big)\dt\bigg].
\end{aligned}
\end{equation*}
It follows from Assumption \eqref{H1} again that $A(i)+A^\top(i)+C^\top(i)C(i)-r\mathbf{I}_n<-m\mathbf{I}_n$
for some $m>0$, and moreover $C^\top(i)C(i)\le \mu\mathbf{I}_n$ for some $\mu>0$. Then,
\begin{equation}\label{X1-Ito-esti}
\begin{aligned}
&E\Big[e^{-rT}|\widecheck{X}(T)|^2\Big]
+\frac{m}{2}E\bigg[\int_0^T e^{-rt}|\widecheck{X}(t)|^2\dt\bigg]\\
\le&\frac{4}{m}E\bigg[\int_0^T e^{-rt}|B(\alpha(t))\widecheck{u}(t)|^2\dt\bigg]
+\bigg(\frac{4\mu}{m}+1\bigg)E\bigg[\int_0^T e^{-rt}|\sigma(t)|^2\dt\bigg].
\end{aligned}
\end{equation}
Combining \eqref{X2-Ito-esti} with \eqref{X1-Ito-esti} and noting \eqref{uhat-wellpose},
there exists a constant $K>0$ such that
$$
E\bigg[\int_0^T e^{-rt}|X(t)|^2\dt\bigg]
\le K\bigg(|x|^2+E\bigg[\int_0^T e^{-rt}|u(t)|^2\dt\bigg]\bigg).
$$
Sending $T\rightarrow\infty$ and applying the monotone convergence theorem, we deduce that
the estimate \eqref{X-wellpose} holds and $X(\cdot)\in L_{\mathcal{F}}^{2,r}(0,\infty;\mathbb{R}^n)$.

Next, we prove the asymptotic property (\ref{Xlim}). Indeed, for any finite $T_2>T_1>0$,
substituting $T$ with $T_1$ and $T_2$ in (\ref{X2-Ito}), we have
\begin{equation*}
\Big|E[e^{-rT_2}|\widehat{X}(T_2)|^2]-E[e^{-rT_1}|\widehat{X}(T_1)|^2]\Big|
\le KE\bigg[\int_{T_1}^{T_2} e^{-rt}\Big(|\widehat{X}(t)|^2+|B(\alpha(t))\widehat{u}(t)|^2\Big)\dt\bigg].
\end{equation*}
In view of $\widehat{X}(\cdot)\in L_{\mathcal{F}}^{2,r}(0,\infty;\mathbb{R}^{n})$ and
$\widehat{u}(\cdot)\in L_{\mathcal{F}}^{2,r}(0,\infty;\mathbb{R}^{k})$, we have the map
$T\mapsto E[e^{-rT}|\widehat{X}(T)|^2]$ is uniformly continuous, and then
$\lim_{T\rightarrow \infty}E[e^{-rT}|\widehat{X}(T)|^2]=0$. Similarly,
we can also derive that $\lim_{T\rightarrow \infty}E[e^{-rT}|\widecheck{X}(T)|^2]=0$.
Then the desired result \eqref{Xlim} follows.
\end{proof}
Next, we turn to the adjoint equation for our discounted optimal control problem.
First, the corresponding Hamiltonian function in a stochastic maximum principle
is given by
\begin{equation}\label{H}
\begin{aligned}
H(x,\widehat{x},i,u,p,q)=\;&\Big\langle A(i)x+\widehat{A}(i)\widehat{x}+B(i)u,p\Big\rangle
+\Big\langle C(i)x+\widehat{C}(i)\widehat{x}+D(i)u,q\Big\rangle\\
&+\frac{1}{2}e^{-rt}\bigg(\Big\langle Q(i)x,x\Big\rangle
+\Big\langle \widehat{Q}(i)\widehat{x},\widehat{x}\Big\rangle
+\Big\langle R(i)u,u\Big\rangle\bigg)-r\langle x,p\rangle.
\end{aligned}
\end{equation}
Based on \eqref{H}, we have the following adjoint equation in the form of an infinite horizon
mean-field BSDE with Markov chain:
\begin{equation}\label{adjoint_PQK}
\begin{aligned}
\dd p(t)=\;&-\Big[A^{\top}(\alpha(t))p(t)+\widehat{A}^{\top}(\alpha(t))\widehat{p}(t)
+C^{\top}(\alpha(t))q(t)+\widehat{C}^{\top}(\alpha(t))\widehat{q}(t)\\
&\qquad+Q(\alpha(t))X(t)+\widehat{Q}(\alpha(t))\widehat{X}(t)-rp(t)\Big ]\dt
+q(t)\dw(t)+\Phi(t)\bullet \dd M(t).
\end{aligned}
\end{equation}
Similarly, we decompose the solution $(p(\cdot),q(\cdot),\Phi(\cdot))$ as follows:
$$
\Gamma(t)
=(\Gamma(t)-E[\Gamma(t)|\mathcal{F}_{t}^{\alpha}])+E[\Gamma(t)|\mathcal{F}_{t}^{\alpha}]
\doteq\widecheck{\Gamma}(t)+\widehat{\Gamma}(t),
$$
for $\Gamma=p,q,\Phi$. Then, the processes $\widecheck{p}(\cdot)$ and $\widehat{p}(\cdot)$ satisfy
$$
\dd \widecheck{p}(t)=-\Big[A^{\top}(\alpha(t))\widecheck{p}(t)+C^{\top}(\alpha(t))\widecheck{q}(t)
-r\widecheck{p}(t)+\widecheck{\varphi}(t)\Big ]\dt+q(t)\dw(t)+\widecheck{\Phi}(t)\bullet \dd M(t),
$$
and
\begin{equation*}
\begin{aligned}
\dd\widehat{p}(t)=\;&-\Big[\Big(A(\alpha(t))+\widehat{A}(\alpha(t))\Big)^{\top}\widehat{p}(t)
+\Big(C(\alpha(t))+\widehat{C}(\alpha(t))\Big)^{\top}\widehat{q}(t)-r\widehat{p}(t)+\widehat{\varphi}(t)\Big ]\dt\\
&+\widehat{\Phi}(t)\bullet \dd M(t),
\end{aligned}
\end{equation*}
where we denote
$$
\varphi(t)=Q(\alpha(t))X(t)+\widehat{Q}(\alpha(t))\widehat{X}(t),
$$
and it follows that
$$
\widehat{\varphi}(t)=[Q(\alpha(t))+\widehat{Q}(\alpha(t))]\widehat{X}(t),
\quad\widecheck{\varphi}(t)=Q(\alpha(t))\widecheck{X}(t).
$$
\begin{lemma}\label{lem:priori}
Let Assumption \eqref{H1} hold. Suppose that there exists a solution $(p(\cdot),q(\cdot),\Phi(\cdot))
\in (L_{\mathcal{F}}^{2,r}(0,\infty;\mathbb{R}^n))^{2}\times M_{\mathcal{F}}^{2,r}(0,\infty;\mathbb{R}^{n})$
to the adjoint equation \eqref{adjoint_PQK}. Then we have
\begin{equation}\label{Plim}
\lim_{T\rightarrow\infty}E[e^{-rT}|p(T)|^2]=0.
\end{equation}
Furthermore, there exists a constant $K>0$, which depends only on $r$ and coefficients
$A$, $\widehat{A}$, $B$, $C$, $\widehat{C}$, $Q$, $\widehat{Q}$, such that
\begin{equation}\label{P-priori}
\begin{aligned}
&E\bigg[|p(0)|^2+\int_0^{\infty} e^{-rt}[|p(t)|^2+|q(t)|^2 ]\dt
+\int_0^{\infty} e^{-rt}|\Phi(t)|^2\bullet \dd\;[M](t)\bigg]\\
\leq&KE\bigg[\int_0^{\infty} e^{-rt}|X(t)|^2\dt\bigg].
\end{aligned}
\end{equation}
\end{lemma}
\begin{proof}
Note that the backward SDE (\ref{adjoint_PQK}) can be rewritten in the form of a forward SDE:
\begin{equation*}
\begin{aligned}
p(t)&=p(0)-\int_0^t\Big[A^{\top}(\alpha(s))p(s)+C^{\top}(\alpha(s))q(s)
+\widehat{A}^{\top}(\alpha(s))\widehat{p}(s)+\widehat{C}^{\top}(\alpha(s))\widehat{q}(s)\\
&\qquad\qquad\qquad-rp(s)+\varphi(s)\Big]\ds+\int_0^t q(s)\dw(s)+\int_0^t \Phi(s)\bullet \dd M(s).
\end{aligned}
\end{equation*}
Using the similar method in Proposition \ref{prop:SDE-wellpose}, for any $T_2>T_1>0$,
there exists some $K>0$ such that
$$
\begin{aligned}
&\Big|E[e^{-rT_2}|p(T_2)|^2]-E[e^{-rT_1}|p(T_1)|^2]\Big|\\
\le&K\bigg[\int_{T_1}^{T_2} e^{-rt}\Big(|p(t)|^2+|q(t)|^2+|\varphi(t)|^2\Big)\dt
+\int_{T_1}^{T_2} e^{-rt}|\Phi(t)|^2\bullet \dd\;[M](t)\bigg].
\end{aligned}
$$
From $p(\cdot),q(\cdot),\varphi(\cdot)\in L_{\mathcal{F}}^{2,r}(0,\infty;\mathbb{R}^{n})$
and $\Phi(\cdot)\in M_{\mathcal{F}}^{2,r}(0,\infty;\mathbb{R}^{n})$, we obtain the uniform continuity
of the map $T\mapsto E[e^{-rT}|p(T)|^2]$ and then the asymptotic property \eqref{Plim}.

Let $\widetilde{C}(i)=C(i)+\widehat{C}(i)$.
Applying It\^{o}'s formula to $e^{-rt}|\widecheck{p}(t)|^2$ and $e^{-rt}|\widehat{p}(t)|^2$,
respectively, we have
\begin{equation}\label{p-estimate(1)}
\begin{aligned}
&E\bigg[|\widecheck{p}(0)|^2+\int_0^T e^{-rt}|q(t)|^2\dt+\int_0^T e^{-rt}|\widecheck{\Phi}(t)|^2\bullet \dd\;[M](t)\bigg]\\
=\;&E\bigg[e^{-rT}|\widecheck{p}(T)|^2
+\int_0^T e^{-rt}\Big[-r|\widecheck{p}|^2+2\Big\langle \widecheck{p},A^{\top}\widecheck{p}
+C^{\top}\widecheck{q}+\widecheck{\varphi}\Big\rangle\Big ]\dt\bigg]\\
\le&E\bigg[e^{-rT}|\widecheck{p}(T)|^2+\int_0^T e^{-rt}\bigg(\frac{1}{1+\varepsilon}|\widecheck{q}|^2
+\frac{2}{m}|\widecheck{\varphi}|^2\bigg)\dt\bigg]\\
&\quad+E\bigg[\int_0^T e^{-rt}\bigg[\bigg\langle\bigg(A+A^{\top}+(1+\varepsilon)C^{\top}C
+\frac{m}{2}\mathbf{I}_n-r\mathbf{I}_n\bigg)\widecheck{p},\widecheck{p}\bigg\rangle\bigg ]\dt\bigg],
\end{aligned}
\end{equation}
and
\begin{equation}\label{p-estimate(2)}
\begin{aligned}
&E\bigg[|\widehat{p}(0)|^2+\int_0^T e^{-rt}|\widehat{\Phi}(t)|^2\bullet \dd\;[M](t)\bigg]\\
=\;&E\bigg[e^{-rT}|\widehat{p}(T)|^2
+\int_0^T e^{-rt}\Big[-r|\widehat{p}|^2+2\Big\langle\widehat{p},\widetilde{A}^{\top}\widehat{p}
+\widetilde{C}^{\top}\widehat{q}+\widetilde{\varphi}\Big\rangle\Big ]\dt\bigg]\\
\le&E\bigg[e^{-rT}|\widehat{p}(T)|^2
+\int_0^T e^{-rt}\bigg\langle\bigg(\widetilde{A}+\widetilde{A}^{\top}
+\frac{m}{2}\mathbf{I}_n-r\mathbf{I}_n\bigg)\widehat{p},\widehat{p}\bigg\rangle\dt\bigg]\\
&\quad+\frac{4}{m}E\bigg[\int_0^T e^{-rt}\Big(|\widetilde{C}^{\top}\widehat{q}|^2
+|\widetilde{\varphi}|^2\Big)\dt\bigg].
\end{aligned}
\end{equation}
Recalling that ${A}(i)+{A}^\top(i)+C^\top(i)C(i)-r\mathbf{I}_n<-m \mathbf{I}_n$
and $C^\top(i)C(i)\le \mu\mathbf{I}_n$ for some positive constants $m$ and $\mu$,
and taking $\varepsilon=\frac{m}{4\mu}$ in \eqref{p-estimate(1)}, we have
\begin{equation}\label{p-estimate(1)-1}
\begin{aligned}
&E\bigg[|\widecheck{p}(0)|^2+\frac{m}{4}\int_0^T e^{-rt}|\widecheck{p}(t)|^2\dt
+\frac{m}{4\mu+m}\int_0^T e^{-rt}|\widecheck{q}(t)|^2\dt+\int_0^T e^{-rt}|\widecheck{\Phi}(t)|^2\bullet \dd\;[M](t)\bigg]\\
\leq&E\bigg[e^{-rT}|\widecheck{p}(T)|^2
+\frac{2}{m}\int_0^T e^{-rt}|\widecheck{\varphi}(t)|^2\dt\bigg].
\end{aligned}
\end{equation}
On the other hand, using the inequality $\widetilde{A}(i)+\widetilde{A}^\top(i)-r\mathbf{I}_n<-m \mathbf{I}_n$
in \eqref{p-estimate(2)}, one has
\begin{equation}\label{p-estimate(2)-2}
\begin{aligned}
&E\bigg[|\widehat{p}(0)|^2+\frac{m}{2}\int_0^T e^{-rt}|\widehat{p}(t)|^2\dt
+\int_0^T e^{-rt}|\widehat{\Phi}(t)|^2\bullet \dd\;[M](t)\bigg]\\
\le&E\bigg[e^{-rT}|\widehat{p}(T)|^2
+\frac{4}{m}\int_0^T e^{-rt}|\widetilde{C}^{\top}(\alpha(t))\widehat{q}(t)|^2\dt
+\frac{4}{m}\int_0^T e^{-rt}|\widehat{\varphi}(t)|^2\dt\bigg].
\end{aligned}
\end{equation}
Combining \eqref{p-estimate(1)-1} and \eqref{p-estimate(2)-2}, there should exist some $K>0$ such that
$$
\begin{aligned}
&E\bigg[|p(0)|^2+\int_0^T e^{-rt}[|p(t)|^2+|q(t)|^2 ]\dt+\int_0^T e^{-rt}|\Phi(t)|^2\bullet \dd\;[M](t)\bigg]\\
\leq&KE\bigg[e^{-rT}|p(T)|^2+\int_0^T e^{-rt}|\varphi(t)|^2\dt\bigg].
\end{aligned}
$$
By sending $T\rightarrow\infty$ and noting that $\lim_{T\rightarrow\infty}E[e^{-rT}|p(T)|^2]=0$
and $E[|\varphi(t)|^2]\le KE[|X(t)|^2]$, we obtain the priori estimate (\ref{P-priori}).
\end{proof}
Based on Lemma \ref{lem:priori}, we have the following existence and uniqueness results
for the adjoint equation (\ref{adjoint_PQK}). The proof is similar to Peng and Shi \cite[Theorem 4]{peng2000infinite}
(see also Yu \cite[Theorem 2.5]{yu2017a}), so it is omitted here for simplicity of presentation.
\begin{theorem}\label{thm:Y-wellpose}
Let Assumption \eqref{H1} hold. Then, the adjoint equation \eqref{adjoint_PQK} admits a unique
solution $(p(\cdot),q(\cdot),\Phi(\cdot))\in (L_{\mathcal{F}}^{2,r}(0,\infty;\mathbb{R}^n))^{2}
\times M_{\mathcal{F}}^{2,r}(0,\infty;\mathbb{R}^{n})$.
\end{theorem}

\section{Formal derivation of algebraic Riccati equations}\label{Section Riccati}

In this section, we present a formal and heuristic derivation for a (candidate) feedback
optimal control and algebraic Riccati equations (AREs) based on stochastic maximum principle
and the adjoint equation \eqref{adjoint_PQK}. We will study the solvability of the AREs
and verify the optimality of the candidate in the next two sections, respectively.
In what follows, for a process $\varphi(\cdot)$, we denote
$$
\widehat{\varphi}(t)=E[\varphi(t)|\mathcal{F}_{t}^{\alpha}].
$$
For convenience, let
\begin{equation*}
\begin{aligned}
\widetilde{R}(i)=\;&R(i)+D^{\top}(i)P(i)D(i),\\
S(i)=\;&B^{\top}(i)P(i)+D^{\top}(i)P(i)C(i),\\
\widehat{S}(i)=\;&B^{\top}(i)\widehat{P}(i)+D^{\top}(i)P(i)\widehat{C}(i).
\end{aligned}
\end{equation*}
From stochastic maximum principle, an optimal control $u^{*}(\cdot)$ necessarily satisfies
\begin{equation}\label{open loop}
\begin{aligned}
R(\alpha(t))u^{*}(t)+B^{\top}(\alpha(t))p(t)+D^{\top}(\alpha(t))q(t)=0.
\end{aligned}
\end{equation}
Inspired by the four-step scheme developed by Ma et al. \cite{MPY1994}
(see also Yong \cite{Yong2013}), we set
\begin{equation}\label{p FSS}
\begin{aligned}
p(t)=P(\alpha(t))X(t)+\widehat{P}(\alpha(t))\widehat{X}(t),
\end{aligned}
\end{equation}
where $P(i)$ and $\widehat{P}(i)$, $i\in\mathcal{M}$, are $(n\times n)$ constant
symmetric matrices. It leads to
\begin{equation}\label{p hat}
\begin{aligned}
\widehat{p}(t)=\Big(P(\alpha(t))+\widehat{P}(\alpha(t))\Big)\widehat{X}(t).
\end{aligned}
\end{equation}
Taking conditional expectation $E[\cdot|\mathcal{F}_{t}^{\alpha}]$ on both sides of
the state equation, we have
\begin{equation*}
\begin{aligned}
\dd\widehat{X}(t)=\;&\Big[\Big(A(\alpha(t))+\widehat{A}(\alpha(t))\Big)\widehat{X}(t)
+B(\alpha(t))\widehat{u}(t)\Big ]\dt.
\end{aligned}
\end{equation*}
Applying It\^{o}'s formula for regime switching diffusions (see Yin and Zhang \cite{YinZhang2013})
to (\ref{p FSS}), we obtain
\begin{equation}\label{dp FSS}
\begin{aligned}
\dd p=\;&\bigg(P\Big[AX+\widehat{A}\widehat{X}+Bu\Big]
+\sum_{j\in\mathcal{M}}\lambda_{\alpha(t),j}P(j)X\bigg)\dt
+P\Big[CX+\widehat{C}\widehat{X}+Du\Big]\dw\\
&+\sum_{i,j\in\mathcal{M}}[P(j)-P(i)]X\dd M_{ij}\\
&+\bigg(\widehat{P}\Big[(A+\widehat{A})\widehat{X}+B\widehat{u}\Big]
+\sum_{j\in\mathcal{M}}\lambda_{\alpha(t),j}\widehat{P}(j)\widehat{X}\bigg)\dt
+\sum_{i,j\in\mathcal{M}}[\widehat{P}(j)-\widehat{P}(i)]\widehat{X}\dd M_{ij}.
\end{aligned}
\end{equation}
Comparing the coefficients of $\dw$ parts in (\ref{adjoint_PQK}) and (\ref{dp FSS}), it follows that
\begin{equation}\label{q}
\begin{aligned}
q=P\Big[CX+\widehat{C}\widehat{X}+Du\Big],
\end{aligned}
\end{equation}
and then,
\begin{equation}\label{q hat}
\begin{aligned}
\widehat{q}=P\Big[(C+\widehat{C})\widehat{X}+D\widehat{u}\Big].
\end{aligned}
\end{equation}
Inserting (\ref{p FSS}) and (\ref{q}) into (\ref{open loop}) yields
\begin{equation*}
\begin{aligned}
0=\;&\Big(R+D^{\top}PD\Big)u^{*}+\Big(B^{\top}P+D^{\top}PC\Big)X
+\Big(B^{\top}\widehat{P}+D^{\top}P\widehat{C}\Big)\widehat{X}\\
=\;&\widetilde{R}u^{*}+SX+\widehat{S}\widehat{X},
\end{aligned}
\end{equation*}
which implies that
\begin{equation}\label{optimal control}
\begin{aligned}
u^{*}=-\widetilde{R}^{-1}[SX+\widehat{S}\widehat{X}],
\end{aligned}
\end{equation}
provided $\widetilde{R}$ is invertible. Then,
\begin{equation}\label{optimal control hat}
\begin{aligned}
\widehat{u}^{*}=-\widetilde{R}^{-1}(S+\widehat{S})\widehat{X}.
\end{aligned}
\end{equation}
On one hand, substituting (\ref{p FSS}), (\ref{p hat}), (\ref{q}), (\ref{q hat}),
(\ref{optimal control}), (\ref{optimal control hat}) into (\ref{adjoint_PQK}),
we have
\begin{equation}\label{comparison 1}
\begin{aligned}
\dd p=\;&-\Big[\Big(A^{\top}P+C^{\top}PC-C^{\top}PD\widetilde{R}^{-1}S+Q-rP\Big)X\\
&\qquad+\Big(\widehat{A}^{\top}P+(A+\widehat{A})^{\top}\widehat{P}
+C^{\top}P\widehat{C}+\widehat{C}^{\top}PC+\widehat{C}^{\top}P\widehat{C}\\
&\qquad-C^{\top}PD\widetilde{R}^{-1}\widehat{S}
-\widehat{C}^{\top}PD\widetilde{R}^{-1}(S+\widehat{S})+\widehat{Q}-r\widehat{P}\Big)\widehat{X}\Big]\dt+(\cdots).
\end{aligned}
\end{equation}
On the other hand, substituting (\ref{optimal control}) and \eqref{optimal control hat}
into (\ref{dp FSS}), we have
\begin{equation}\label{comparison 2}
\begin{aligned}
\dd p=\;&\bigg[\bigg(PA-PB\widetilde{R}^{-1}S+\sum_{j\in\mathcal{M}}\lambda_{\alpha(t),j}P(j)\bigg)X\\
&\quad+\bigg(P\widehat{A}+\widehat{P}(A+\widehat{A})
-PB\widetilde{R}^{-1}\widehat{S}-\widehat{P}B\widetilde{R}^{-1}(S+\widehat{S})
+\sum_{j\in\mathcal{M}}\lambda_{\alpha(t),j}\widehat{P}(j)\bigg)\widehat{X}\bigg]\dt+(\cdots).
\end{aligned}
\end{equation}
Note that, for simplicity, only the drift terms are presented in the above two equations.
Finally, by equalizing the coefficients of $X$ and $\widehat{X}$ in (\ref{comparison 1})
and (\ref{comparison 2}), we obtain the following two algebraic Riccati equations:
\begin{equation}\label{Riccati equation 1}
\left\{
\begin{aligned}
rP(i)=\;&P(i)A(i)+A^{\top}(i)P(i)+C^{\top}(i)P(i)C(i)+Q(i)\\
&-S^{\top}(i)\widetilde{R}^{-1}(i)S(i)+\sum_{j\in\mathcal{M}}\lambda_{ij}P(j),\\
\widetilde{R}(i)=\;&R(i)+D^{\top}(i)P(i)D(i)>0,
\end{aligned}
\right.
\end{equation}
and
\begin{equation}\label{Riccati equation 2}
\begin{aligned}
r\widehat{P}(i)=\;&\widehat{P}(i)(A(i)+\widehat{A}(i))
+(A(i)+\widehat{A}(i))^{\top}\widehat{P}(i)
+P(i)\widehat{A}(i)+\widehat{A}^{\top}(i)P(i)\\
&+C^{\top}(i)P(i)\widehat{C}(i)+\widehat{C}^{\top}(i)P(i)C(i)
+\widehat{C}^{\top}(i)P(i)\widehat{C}(i)+\widehat{Q}(i)\\
&-S^{\top}(i)\widetilde{R}^{-1}(i)\widehat{S}(i)
-\widehat{S}^{\top}(i)\widetilde{R}^{-1}(i)S(i)
-\widehat{S}^{\top}(i)\widetilde{R}^{-1}(i)\widehat{S}(i)
+\sum_{j\in\mathcal{M}}\lambda_{ij}\widehat{P}(j).
\end{aligned}
\end{equation}
Note that in the first ARE, there is an additional positive definite constraint
$\widetilde{R}>0$, which contains the solution $P$ and becomes part of the equation,
and the second ARE depends on the first ARE. In particular, these two AREs are
both \emph{self-coupled} via the generator of the Markov chain.

Further, let $\widetilde{P}(i)=P(i)+\widehat{P}(i)$, then we can see that $\widetilde{P}$
satisfies
\begin{equation}\label{Riccati equation 3}
\begin{aligned}
r\widetilde{P}(i)
=\;&\widetilde{P}(i)\widetilde{A}(i)+\widetilde{A}^{\top}(i)\widetilde{P}(i)
+\widetilde{C}^{\top}(i)P(i)\widetilde{C}(i)+\widetilde{Q}(i)-\widetilde{S}^{\top}(i)\widetilde{R}^{-1}(i)\widetilde{S}(i)
+\sum_{j\in\mathcal{M}}\lambda_{ij}\widetilde{P}(j),
\end{aligned}
\end{equation}
where $\widetilde{\Lambda}\doteq\Lambda+\widehat{\Lambda}$ for $\Lambda=A,C,Q,S$.
The structure of equation (\ref{Riccati equation 3}) is similar to that of
equation (\ref{Riccati equation 1}) and much simpler than that of equation (\ref{Riccati equation 2}).
So, in the following analysis, we may use equations (\ref{Riccati equation 1})
and (\ref{Riccati equation 3}) instead of (\ref{Riccati equation 1}) and (\ref{Riccati equation 2}).

\begin{remark}
It is interesting to note that an alternative form for the ansatz \eqref{p FSS}, which is inspired
by the traditional mean-field LQ control problem with no regime switching studied in Yong \cite{Yong2013},
is given by
\begin{equation*}
\begin{aligned}
p(t)=\overline{P}(\alpha(t))(X(t)-E[X(t)|\mathcal{F}_t^{\alpha}])
+\widetilde{P}(\alpha(t))E[X(t)|\mathcal{F}_t^{\alpha}].
\end{aligned}
\end{equation*}
It can be verified that the two forms are equivalent under the transformation
$(\overline{P},\widetilde{P})=(P,P+\widehat{P})$. Consequently, if we adopt the above alternative form,
then $\overline{P}$ coincides with $P$ in ARE \eqref{Riccati equation 1}, while $\widetilde{P}$
corresponds to $P+\widehat{P}$ in ARE \eqref{Riccati equation 3}.
\end{remark}

\section{Solvability of the algebraic Riccati equations}\label{Section Riccati solv}

In this section, we focus on the solvability of the AREs \eqref{Riccati equation 1}
and \eqref{Riccati equation 3}. For convenience, we denote $\mathcal{M}(\mathbb{S}_+^n)$ the set
of tuples $P=(P(1),\ldots,P(M))$, where $P(i)\in\mathbb{S}_+^n$, $i\in\mathcal{M}$. As a preliminary,
we first establish the well-posedness of the following Lyapunov equation \eqref{Lyapunov-1} by
a kind of \emph{Feynman-Kac representation}.
\begin{lemma}\label{lem:Lyapunov}
Let Assumption \eqref{H1} hold. Then, for any $Q\in\mathcal{M}(\mathbb{S}_+^n)$,
the following Lyapunov equation:
\begin{equation}\label{Lyapunov-1}
P(i)A(i)+A^{\top}(i)P(i)+C^{\top}(i)P(i)C(i)+Q(i)+\sum_{j\in\mathcal{M}}\lambda_{ij}P(j)-rP(i)=0
\end{equation}
admits a unique solution $P\in\mathcal{M}(\mathbb{S}_+^n)$.
\end{lemma}
\begin{proof}
Consider the following matrix-valued equation on $[0,\infty)$:
\begin{equation*}
\left\{
\begin{aligned}
\dd\Phi_i(t)=\;&A(\alpha(t))\Phi_i(t)\dt+C(\alpha(t))\Phi_i(t)\dw(t),\\
\Phi_i(0)=\;&\mathbf{I}_n,\quad\alpha(0)=i.
\end{aligned}
\right.
\end{equation*}
Similar to Proposition \ref{prop:SDE-wellpose}, the above equation admits a unique solution
$\Phi_i(\cdot)\in L_{\mathcal{F}}^{2,r}(0,\infty;\mathbb{R}^{n\times n})$ under Assumption \eqref{H1}.
Then, as the method in Zhang et al. \cite[Proposition 3.2]{zhang2021}, the (unique) solution to
the Lyapunov equation \eqref{Lyapunov-1} admits the following representation:
$$
P(i)=\mathbb{E}\bigg[\int_0^{\infty}e^{-rt}\Phi_i^{\top}(t)Q(\alpha(t))\Phi_i(t)\dt
\;\bigg|\;\alpha(0)=i\bigg].
$$
Moreover, $Q\in\mathcal{M}(\mathbb{S}_+^n)$ suggests that $P\in\mathcal{M}(\mathbb{S}_+^n)$.
\end{proof}
\begin{theorem}\label{thm:Riccati1}
Let Assumptions \eqref{H1}-\eqref{H2} hold. Then, the algebraic Riccati equation \eqref{Riccati equation 1}
admits a solution $P\in\mathcal{M}(\mathbb{S}_+^n)$.
\end{theorem}
\begin{proof}
First, consider the following Lyapunov equation:
\begin{equation}\label{Lya-1}
P_0(i)A(i)+A^{\top}(i)P_0(i)+C^{\top}(i)P_0(i)C(i)+Q(i)
+\sum_{j\in\mathcal{M}}\lambda_{ij}P_0(j)-rP_0(i)=0.
\end{equation}
From Lemma \ref{lem:Lyapunov}, the equation \eqref{Lya-1} admits a unique solution
$P_0\in\mathcal{M}(\mathbb{S}_+^n)$. Then, let $P_{-1}=0$ (it is defined since we
shall use $\Theta_{-1}$, $A_{-1}$ etc below), and for $k=-1,0,1,\ldots$,
we define inductively that
\begin{equation}\label{theta-def}
\left\{
\begin{aligned}
\Theta_k=\;&-(R+D^{\top}P_kD)^{-1}(B^{\top}P_k+D^{\top}P_kC),\\
A_k=\;&A+B\Theta_k,\quad C_k=C+D\Theta_k,\quad Q_k=Q+\Theta_k^{\top}R\Theta_k.
\end{aligned}
\right.
\end{equation}
By Lemma \ref{lem:Lyapunov}, there exists a unique $P_{k+1}\in\mathcal{M}(\mathbb{S}_+^n)$ satisfying
the following Lyapunov equation:
\begin{equation}\label{Lya-k}
\begin{aligned}
P_{k+1}(i)A_k(i)+A_k^{\top}(i)P_{k+1}(i)+C_k^{\top}(i)P_{k+1}(i)C_k(i)+Q_k(i)
+\sum_{j\in \mathcal{M}}{\lambda _{ij}P_{k+1}(j)}-rP_{k+1}(i)=0.
\end{aligned}
\end{equation}
We claim that the sequence $P_{k}$ decreases to a limit $P$. To show this, we denote
$$
\Delta_k(i)=P_k(i)-P_{k+1}(i),\quad \Lambda_k(i)=\Theta_{k-1}(i)-\Theta_k(i).
$$
Then, for $k\ge 0$, we have
$$
\begin{aligned}
P_k(i)A_{k-1}(i)+A_{k-1}^{\top}(i)P_k(i)&+C_{k-1}^{\top}(i)P_k(i)C_{k-1}(i)+Q_{k-1}(i)
+\sum_{j\in \mathcal{M}}{\lambda_{ij}\Delta_k(j)}-r\Delta_k(i)\\
&-P_{k+1}(i)A_k(i)-A_k^{\top}(i)P_{k+1}(i)-C_k^{\top}(i)P_{k+1}(i)C_k(i)-Q_k(i)=0,
\end{aligned}
$$
i.e.,
\begin{equation}\label{Delta-k}
\begin{aligned}
&\Delta_k(i)A_k(i)+A_k^{\top}(i)\Delta_k(i)+C_k^{\top}(i)\Delta_k(i)C_k(i)
+\sum_{j\in \mathcal{M}}{\lambda_{ij}\Delta_k(j)}-r\Delta_k(i)\\
&+P_k(i)(A_{k-1}(i)-A_k(i))+(A_{k-1}(i)-A_k(i))^{\top}P_{k}(i)\\
&+C_{k-1}^{\top}(i)P_k(i)C_{k-1}(i)-C_k^{\top}(i)P_k(i)C_k(i)+Q_{k-1}(i)-Q_k(i)=0.
\end{aligned}
\end{equation}
Using the notations given by \eqref{theta-def}, we have
\begin{equation}\label{relationship-k}
\begin{cases}
A_{k-1}-A_k=B\Lambda_k,\quad C_{k-1}-C_k=D\Lambda_k,\\
C_{k-1}^{\top}P_kC_{k-1}-C_{k}^{\top}P_kC_k
=C_{k}^{\top}P_kD\Lambda_k+\Lambda _{k}^{\top}D^{\top}P_kC_k+\Lambda_{k}^{\top}D^{\top}P_kD\Lambda_k,\\
Q_{k-1}-Q_k=\Theta_{k}^{\top}R\Lambda_k+\Lambda_{k}^{\top}R\Theta_k+\Lambda_{k}^{\top}R\Lambda_k.
\end{cases}
\end{equation}
By \eqref{theta-def}, one also has
$$
B^{\top}P_k+D^{\top}P_kC_k+R\Theta_k=B^{\top}P_k+D^{\top}P_kC+(R+D^{\top}P_kD)\Theta_k=0,
$$
so substituting \eqref{relationship-k} into \eqref{Delta-k} yields
$$
\begin{aligned}
\Delta_k(i)A_k(i)+A_k^{\top}(i)\Delta_k(i)+C_k^{\top}(i)\Delta_k(i)C_k(i)
+\sum_{j\in \mathcal{M}}{\lambda_{ij}\Delta_k(j)}-r\Delta_k(i)&\\
+\Lambda_k^{\top}(i)[R(i)+D^{\top}(i)P_k(i)D(i)]\Lambda_k(i)&=0.
\end{aligned}
$$
From Lemma \ref{lem:Lyapunov} and $\Lambda_k^{\top}[R+D^{\top}P_kD]\Lambda_k\ge 0$,
the above equation admits a unique solution $\Delta_k\in\mathcal{M}(\mathbb{S}_+^n)$,
which indicates that
$$
P_0\geq P_1\geq\cdots\geq P_k\geq P_{k+1}\geq\cdots\geq0.
$$
By the monotone convergence theorem, $P_{k}$ should converge to a limit $P\in\mathcal{M}(\mathbb{S}_+^n)$.
Moreover,
$$
R+D^{\top}PD=\lim_{k\rightarrow\infty}[R+D^{\top}P_kD]\ge R>0,
$$
and as $k\rightarrow\infty$,
\begin{equation*}
\left\{
\begin{aligned}
\Theta_k\rightarrow&-(R+D^{\top}PD)^{-1}(B^{\top}P+D^{\top}PC)\doteq\Theta,\\
A_k\rightarrow&A+B\Theta,\quad C_k\rightarrow C+D\Theta,\quad Q_k\rightarrow Q+\Theta^{\top}R\Theta.
\end{aligned}
\right.
\end{equation*}
Sending $k\rightarrow\infty$ in \eqref{Lya-k}, it follows that $P$ satisfies the following equation:
$$
\begin{aligned}
P(A+B\Theta)+(A+B\Theta)^{\top}P+\sum_{j\in \mathcal{M}}{\lambda_{ij}P(j)}-rP
+(C+D\Theta)^{\top}P(C+D\Theta)+Q+\Theta^{\top}R\Theta=0,
\end{aligned}
$$
which coincides with the ARE (\ref{Riccati equation 1}).
\end{proof}
\begin{theorem}
Let Assumptions \eqref{H1}-\eqref{H2} hold. Then, the algebraic Riccati equation \eqref{Riccati equation 3}
admits a solution $\widetilde{P}\in\mathcal{M}(\mathbb{S}_+^n)$.
\end{theorem}
\begin{proof}
Given a solution $P\in\mathcal{M}(\mathbb{S}_+^n)$ to the ARE (\ref{Riccati equation 1}), we denote
$\overline{C}=D^{\top}P\widetilde{C}$, $\overline{Q}=\widetilde{C}^{\top}P\widetilde{C}+\widetilde{Q}\ge 0$,
$\widetilde{S}=B^{\top}\widetilde{P}+\overline{C}$. Then, the ARE (\ref{Riccati equation 3}) becomes
$$
\widetilde{P}\widetilde{A}+\widetilde{A}^{\top}\widetilde{P}+\overline{Q}
+\sum_{j\in \mathcal{M}}{\lambda_{ij}\widetilde{P}(j)}
-\widetilde{S}^{\top}\widetilde{R}^{-1}\widetilde{S}-r\widetilde{P}=0.
$$
We first consider the following equation:
$$
\widetilde{P}_0\widetilde{A}+\widetilde{A}^{\top}\widetilde{P}_0+\overline{Q}
+\sum_{j\in \mathcal{M}}\lambda_{ij}\widetilde{P}_0(j)-r\widetilde{P}_0=0,
$$
which by Lemma \ref{lem:Lyapunov} admits a unique solution $\widetilde{P}_0\in\mathcal{M}(\mathbb{S}_+^n)$
under Assumptions \eqref{H1}-\eqref{H2}. Next, for $k=0,1,\ldots$, let $\widetilde{P}_{k+1}$
be the solution to the following equation:
$$
\widetilde{P}_{k+1}\widetilde{A}_k+\widetilde{A}_k^{\top}\widetilde{P}_{k+1}
+\overline{Q}_k+\sum_{j\in\mathcal{M}}{\lambda_{ij}\widetilde{P}_{k+1}(j)}-r\widetilde{P}_{k+1}=0,
$$
where we denote
$\widetilde{\Theta}_k=-\widetilde{R}^{-1}(B^{\top}\widetilde{P}_k+\overline{C})$,
$\widetilde{A}_k=\widetilde{A}+B\widetilde{\Theta}_k$,
$\overline{Q}_k=\overline{Q}+\widetilde{\Theta}_{k}^{\top}\widetilde{R}\widetilde{\Theta}_k$.
By the argument analogous to Theorem \ref{thm:Riccati1}, the sequence $\widetilde{P}_{k}$
converges to a limit $\widetilde{P}$, which solves the ARE (\ref{Riccati equation 3}).
\end{proof}

\section{Verification for optimality}\label{Section optimality}

In the following, we will verify the optimality of the candidate (\ref{optimal control})
and compute the minimal cost based on the completion of squares method.
\begin{theorem}\label{case 1 theorem}
Let Assumptions \eqref{H1}-\eqref{H2} hold. Suppose that the algebraic Riccati equations
\eqref{Riccati equation 1} and \eqref{Riccati equation 3}
admit solutions $P(i)$ and $\widetilde{P}(i)$, $i\in\mathcal{M}$, respectively. Then, $u^{*}(\cdot)$
defined by \eqref{optimal control} is a feedback optimal control. Furthermore, the minimal cost
is given by
\begin{equation*}
\begin{aligned}
J(x,i;u^{*}(\cdot))=\frac{1}{2}\langle \widetilde{P}(i)x,x\rangle.
\end{aligned}
\end{equation*}
As a consequence, the solutions to the ARE \eqref{Riccati equation 1} and \eqref{Riccati equation 3} are unique.
\end{theorem}
\begin{proof}
In order to use the AREs (\ref{Riccati equation 1}) and (\ref{Riccati equation 3})
rather than (\ref{Riccati equation 1}) and (\ref{Riccati equation 2}), we first
rewrite the cost functional (\ref{cost functional}) as
\begin{equation*}
\begin{aligned}
J(x,i;u(\cdot))
=\;&\frac{1}{2}E\bigg[\int_{0}^{\infty}e^{-rt}\Big(\langle QX,X\rangle
+\langle \widehat{Q}\widehat{X},\widehat{X}\rangle
+\langle Ru,u\rangle\Big)\dt\bigg]\\
=\;&\frac{1}{2}E\bigg[\int_{0}^{\infty}e^{-rt}\Big(\langle Q(X-\widehat{X}),X-\widehat{X}\rangle
+\langle \widetilde{Q}\widehat{X},\widehat{X}\rangle
+\langle Ru,u\rangle\Big)\dt\bigg].
\end{aligned}
\end{equation*}
Then, using the fact that $X(0)=\widehat{X}(0)=x$, we have
\begin{equation}\label{completion of squares 1}
\begin{aligned}
J(x,i;u(\cdot))=\;&J(x,i;u(\cdot))+\frac{1}{2}\langle \widetilde{P}(i)\widehat{X}(0),\widehat{X}(0)\rangle\\
&-\frac{1}{2}\langle P(i)(X(0)-\widehat{X}(0)),X(0)-\widehat{X}(0)\rangle
-\frac{1}{2}\langle \widetilde{P}(i)\widehat{X}(0),\widehat{X}(0)\rangle\\
=\;&J(x,i;u(\cdot))
+\frac{1}{2}\langle \widetilde{P}(i)x,x\rangle\\
&+\frac{1}{2}E\bigg[\int_{0}^{\infty}\dd e^{-rt}\langle P(X-\widehat{X}),X-\widehat{X}\rangle\bigg]
+\frac{1}{2}E\bigg[\int_{0}^{\infty}\dd e^{-rt}\langle \widetilde{P}\widehat{X},\widehat{X}\rangle\bigg]\\
=\;&\frac{1}{2}\langle \widetilde{P}(i)x,x\rangle
+\frac{1}{2}E\bigg[\int_{0}^{\infty}e^{-rt}\Big(\langle Q(X-\widehat{X}),X-\widehat{X}\rangle
+\langle \widetilde{Q}\widehat{X},\widehat{X}\rangle
+\langle Ru,u\rangle\Big)\dt\bigg]\\
&+\frac{1}{2}E\bigg[\int_{0}^{\infty}\Big(\dd e^{-rt}\langle P(X-\widehat{X}),X-\widehat{X}\rangle
+\dd\;\langle e^{-rt}\widetilde{P}\widehat{X},\widehat{X}\rangle\Big)\bigg].
\end{aligned}
\end{equation}
Note that
\begin{equation}\label{X-Xhat}
\begin{aligned}
\dd\;(X-\widehat{X})
=\;&[A(X-\widehat{X})+Bu-B\widehat{u} ]\dt+[CX+\widehat{C}\widehat{X}+Du]\dw\\
=\;&[A(X-\widehat{X})+Bu-B\widehat{u} ]\dt+[C(X-\widehat{X})+\widetilde{C}\widehat{X}+Du]\dw.
\end{aligned}
\end{equation}
On one hand, applying It\^{o}'s formula for regime switching diffusions to $P(X-\widehat{X})$,
it follows that
\begin{equation}\label{P X-Xhat}
\begin{aligned}
\dd\;[P(X-\widehat{X})]
=\;&\Big(P[A(X-\widehat{X})+Bu-B\widehat{u}]
+\sum_{j\in\mathcal{M}}\lambda_{\alpha(t),j}P(j)(X-\widehat{X})\Big)\dt\\
&+P[C(X-\widehat{X})+\widetilde{C}\widehat{X}+Du]\dw
+\sum_{i,j\in\mathcal{M}}[P(j)-P(i)](X-\widehat{X})\dd M_{ij}.
\end{aligned}
\end{equation}
Applying It\^{o}'s formula for semi-martingales to $e^{-rt}\langle P(X-\widehat{X}),X-\widehat{X}\rangle$
(only the drift terms are presented below), we get
\begin{equation}\label{P X-Xhat X-Xhat 1}
\begin{aligned}
&\dd e^{-rt}\langle P(X-\widehat{X}),X-\widehat{X}\rangle\\
=\;&e^{-rt}\bigg[-r\langle P(X-\widehat{X}),X-\widehat{X}\rangle\\
&\qquad+\Big\langle P[A(X-\widehat{X})+Bu-B\widehat{u}]
+\sum_{j\in\mathcal{M}}\lambda_{\alpha(t),j}P(j)(X-\widehat{X}),X-\widehat{X}\Big\rangle\\
&\qquad+\Big\langle P(X-\widehat{X}),A(X-\widehat{X})+Bu-B\widehat{u}\Big\rangle\\
&\qquad+\Big\langle P[C(X-\widehat{X})+\widetilde{C}\widehat{X}+Du],
C(X-\widehat{X})+\widetilde{C}\widehat{X}+Du\Big\rangle\bigg ]\dt+(\cdots),
\end{aligned}
\end{equation}
i.e.,
\begin{equation}\label{P X-Xhat X-Xhat 2}
\begin{aligned}
&\dd e^{-rt}\langle P(X-\widehat{X}),X-\widehat{X}\rangle\\
=\;&e^{-rt}\bigg[\Big\langle \Big[-rP+PA+A^{\top}P+C^{\top}PC+\sum_{j\in\mathcal{M}}\lambda_{\alpha(t),j}P(j)\Big]
(X-\widehat{X}),X-\widehat{X}\Big\rangle\\
&\qquad+2\langle PBu,X-\widehat{X}\rangle+\langle D^{\top}PDu,u\rangle\\
&\qquad+2\langle PDu,C(X-\widehat{X})+\widetilde{C}\widehat{X}\rangle
+\langle P\widetilde{C}\widehat{X},\widetilde{C}\widehat{X}\rangle\bigg ]\dt+(\cdots).
\end{aligned}
\end{equation}
On the other hand, applying It\^{o}'s formula for regime switching diffusions to $\widetilde{P}\widehat{X}$ yields
\begin{equation}\label{P tilde X hat}
\begin{aligned}
\dd\widetilde{P}\widehat{X}
=\Big(\widetilde{P}[(A+\widehat{A})\widehat{X}+B\widehat{u}]
+\sum_{j\in\mathcal{M}}\lambda_{\alpha(t),j}\widehat{P}(j)\widehat{X}\Big)\dt
+\sum_{i,j\in\mathcal{M}}[\widetilde{P}(j)-\widetilde{P}(i)]\widehat{X}\dd M_{ij}.
\end{aligned}
\end{equation}
Applying It\^{o}'s formula for semi-martingales to
$e^{-rt}\langle \widetilde{P}\widehat{X},\widehat{X}\rangle$, one has
\begin{equation}\label{P tilde X hat X hat 1}
\begin{aligned}
\dd e^{-rt}\langle \widetilde{P}\widehat{X},\widehat{X}\rangle
=\;&e^{-rt}\bigg[-r\langle \widetilde{P}\widehat{X},\widehat{X}\rangle
+\Big\langle \Big(\widetilde{P}[(A+\widehat{A})\widehat{X}+B\widehat{u}]
+\sum_{j\in\mathcal{M}}\lambda_{\alpha(t),j}\widetilde{P}(j)\widehat{X}\Big),\widehat{X}\Big\rangle\\
&\qquad+\Big\langle \widetilde{P}\widehat{X},(A+\widehat{A})\widehat{X}+B\widehat{u}\Big\rangle\bigg ]\dt+(\cdots),
\end{aligned}
\end{equation}
i.e.,
\begin{equation}\label{P tilde X hat X hat 2}
\begin{aligned}
\dd e^{-rt}\langle \widetilde{P}\widehat{X},\widehat{X}\rangle
=\;&e^{-rt}\bigg[\Big\langle\Big(-r\widetilde{P}+\widetilde{P}(A+\widehat{A})+(A+\widehat{A})^{\top}\widetilde{P}
+\sum_{j\in\mathcal{M}}\lambda_{\alpha(t),j}\widetilde{P}(j)\Big)\widehat{X},\widehat{X}\Big\rangle\\
&\qquad+2\Big\langle \widetilde{P}B\widehat{u},\widehat{X}\Big\rangle\bigg ]\dt+(\cdots).
\end{aligned}
\end{equation}
We first look at the drift terms involving $u$ and $\widehat{u}$, using (\ref{P X-Xhat}),
(\ref{P X-Xhat X-Xhat 2}), (\ref{P tilde X hat}), (\ref{P tilde X hat X hat 2}):
\begin{equation*}
\begin{aligned}
&u^{\top}(R+D^{\top}PD)u+2u^{\top}[B^{\top}P(X-\widehat{X})
+D^{\top}P(C(X-\widehat{X})+\widetilde{C}\widehat{X})+B^{\top}\widetilde{P}\widehat{X}]\\
=\;&u^{\top}\widetilde{R}u+2u^{\top}[S(X-\widehat{X})+\widetilde{S}\widehat{X}]\\
=\;&u^{\top}\widetilde{R}u
+2u^{\top}\widetilde{R}\widetilde{R}^{-1}[S(X-\widehat{X})+\widetilde{S}\widehat{X}]
+|\widetilde{R}^{-\frac{1}{2}}[S(X-\widehat{X})+\widetilde{S}\widehat{X}]|^{2}\\
&-|\widetilde{R}^{-\frac{1}{2}}[S(X-\widehat{X})+\widetilde{S}\widehat{X}]|^{2}\\
=\;&(u+\widetilde{R}^{-1}[S(X-\widehat{X})+\widetilde{S}\widehat{X}])^{\top}
\widetilde{R}(u+\widetilde{R}^{-1}[S(X-\widehat{X})+\widetilde{S}\widehat{X}])\\
&-|\widetilde{R}^{-\frac{1}{2}}[S(X-\widehat{X})+\widetilde{S}\widehat{X}]|^{2},
\end{aligned}
\end{equation*}
in which we have used
\begin{equation*}
\begin{aligned}
&E[\langle PB\widehat{u},X-\widehat{X}\rangle]
=E[\langle PBu,\widehat{X}-\widehat{X}\rangle]=0,\\
&E[\langle \widetilde{P}B\widehat{u},\widehat{X}\rangle]
=E[\langle \widetilde{P}Bu,\widehat{X}\rangle].
\end{aligned}
\end{equation*}
For the remainder terms, one has
\begin{equation*}
\begin{aligned}
&\langle(-Q+S^{\top}\widetilde{R}^{-1}S)(X-\widehat{X}),X-\widehat{X}\rangle
+\langle PC(X-\widehat{X}),\widetilde{C}\widehat{X}\rangle
+\langle P\widetilde{C}\widehat{X},\widetilde{C}\widehat{X}\rangle\\
&+\langle(-\widetilde{Q}-\widetilde{C}^{\top}P\widetilde{C}
+\widetilde{S}^{\top}\widetilde{R}^{-1}\widetilde{S})\widehat{X},\widehat{X}\rangle\\
=\;&-\langle Q(X-\widehat{X}),X-\widehat{X}\rangle
-\langle\widetilde{Q}\widehat{X},\widehat{X}\rangle
+|\widetilde{R}^{-\frac{1}{2}}[S(X-\widehat{X})+\widetilde{S}\widehat{X}]|^{2},
\end{aligned}
\end{equation*}
in which we have used
\begin{equation*}
\begin{aligned}
&E[\langle PC(X-\widehat{X}),\widetilde{C}\widehat{X}\rangle]
=E[\langle PC(\widehat{X}-\widehat{X}),\widetilde{C}\widehat{X}\rangle]=0,\\
&E[\langle \widetilde{R}^{-1}S(X-\widehat{X}),\widetilde{S}\widehat{X}\rangle]
=E[\langle \widetilde{R}^{-1}S(\widehat{X}-\widehat{X}),\widetilde{S}\widehat{X}\rangle]=0.
\end{aligned}
\end{equation*}
Finally, we arrive at
\begin{equation}\label{completion of squares 2}
\begin{aligned}
J(x,i;u(\cdot))
=\frac{1}{2}\langle \widetilde{P}(i)x,x\rangle
+E\bigg[\int_{0}^{\infty}e^{-rt}
|\widetilde{R}^{\frac{1}{2}}(u+\widetilde{R}^{-1}[S(X-\widehat{X})+\widetilde{S}\widehat{X}])|^{2}\dt\bigg].
\end{aligned}
\end{equation}
Then we deduce that $u^{*}(\cdot)$ defined by (\ref{optimal control}) is an optimal feedback control
and the minimum cost is given by
\begin{equation*}
\begin{aligned}
J(x,i;u^{*}(\cdot))=\frac{1}{2}\langle \widetilde{P}(i)x,x\rangle.
\end{aligned}
\end{equation*}
Since the optimal value is unique, the above implies $\widetilde{P}$ is the unique solution
to the ARE \eqref{Riccati equation 3}.

Other other hand, since our control problem admits at most one optimal control, the optimal
feedback control $u^{*}(\cdot)$ defined by (\ref{optimal control}) is unique, which further
implies that ${P}$ is the unique solution to the ARE \eqref{Riccati equation 1}.
The proof is now completed.
\end{proof}

\section{Numerical experiments}\label{Section Numerical experiments}

In this section, we provide a numerical example to demonstrate our theoretical results obtained
in the previous sections.
Let $n=k=1$ and $r=3$. Consider the state equation \eqref{state equation} involving
a four-state Markov chain $\alpha$ (i.e., $\mathcal{M}=\{1,2,3,4\}$), whose generator
is given by
$$
Q=\left[
\begin{matrix}
	-3&		2&		1&		0\\
	2&		-3&		0&		1\\
	1&		0&		-3&		2\\
	0&		1&		2&		-3\\
\end{matrix}
\right].
$$
The generator $Q$ means that the Markov chain jumps within the groups $\{1,2\}$ and $\{3,4\}$
with transition rate 2, whereas less frequently between the two groups with transition rate 1.
In addition, the values of the coefficients of state equation \eqref{state equation}
and cost functional \eqref{cost functional} are given in Table \ref{tab:model_parameters_states}.
\begin{table}[htbp]
\centering
\renewcommand\arraystretch{1.5}
\begin{tabular}{cccccccccc}
\hline
& $A(i)$ & $\widehat{A}(i)$ & $B(i)$ & $C(i)$ & $\widehat{C}(i)$ & $D(i)$ & $Q(i)$ & $\widehat{Q}(i)$ & $R(i)$\\
\hline
$i=1$ & 1 & 0 & 1 & 0.5 & 0.2 & 0.4 & 1 & 1 & 0.5 \\
$i=2$ & 1 & 0 & 1.5 & 0.5 & 0.5 & 0.6 & 0.5 & 2 & 0.5 \\
$i=3$ & -1 & 1 & 2 & 0.5 & 0.2 & 0.4 & 1 & 2 & 0.5 \\
$i=4$ & -1 & 1 & 2.5 & 0.5 & 0.5 & 0.6 & 0.5 & 1 & 0.5 \\
\hline
\end{tabular}
\caption{Model parameters for different states.}
\label{tab:model_parameters_states}
\end{table}

In view of that the optimal control (\ref{optimal control}) depends only on the solutions $P$
and $\widetilde{P}$ to Riccati equations (\ref{Riccati equation 1}) and (\ref{Riccati equation 3}),
respectively. So, in order to implement our control policy in practice, the whole task for us
is to compute (\ref{Riccati equation 1}) and (\ref{Riccati equation 3}).
Now, we first introduce the concept of residual error for numerical solutions to AREs.
Taking the ARE (\ref{Riccati equation 1}) as an example. Let $P=(P(1),\ldots,P(M))$ denote
a numerical solution to (\ref{Riccati equation 1}). The residual for each $i\in\mathcal{M}$
is defined as
\begin{equation*}
\begin{aligned}
E_i(P)\doteq P(i)A(i)+A^\top(i)P(i)+C^\top(i)P(i)C(i)+Q(i)+\sum_{j\in\mathcal{M}}\lambda_{ij}P(j)
-S^\top(i)\widetilde{R}^{-1}(i)S(i)-rP(i).
\end{aligned}
\end{equation*}
Then, the residual error is calculated as $\|E_i(P)\|$, where $\|\cdot\|$ represents the Frobenius
norm in $\mathbb{S}^n$.

The procedure for solving the ARE (28), based on the so-called quasi-linearization method
via Lyapunov equations, is given by:
\begin{itemize}
\item Step 1: Solve the initial Lyapunov equation \eqref{Lya-1} to obtain the initial solution $P_0$;
\item Step 2: For each iteration $k$, update $\Theta_k$, $A_k$, $C_k$, and $Q_k$ as in equation \eqref{theta-def},
then solve the Lyapunov equation \eqref{Lya-k} to get $P_{k+1}$;
\item Step 3: At each iteration $k$, compute the residual error $\|E_i(P_k)\|$; the iteration stops
when $\|E_i(P_k)\|\leq 1\times 10^{-10}$ for all $i\in\mathcal{M}$.
\end{itemize}

Using this method, we solve ARE (\ref{Riccati equation 1}) and obtain a positive solution
$P(1)=0.361$, $P(2)=0.259$, $P(3)=0.171$, $P(4)=0.117$ with residual 1e-10.
Then, we solve ARE (\ref{Riccati equation 3}) and obtain a positive solution $\widetilde{P}(1)=0.687$,
$\widetilde{P}(2)=0.617$, $\widetilde{P}(3)=0.448$, $\widetilde{P}(4)=0.307$ with residual 1e-10.
Hence, the optimal feedback control law (\ref{optimal control}) is given by
\begin{equation*}
\begin{aligned}
u^{*}(x,i)=-\widetilde{R}^{-1}(i)\left[S(i)x+\widehat{S}(i)\widehat{x}\right]
\doteq\Theta(i)x+\widehat{\Theta}(i)\widehat{x},
\end{aligned}
\end{equation*}
with
$$
\left[
\begin{array}{c}
	\Theta(1)\\
	\Theta(2)\\
	\Theta(3)\\
	\Theta(4)\\
\end{array}
\right]
=\left[
\begin{array}{c}
	-0.776\\
	-0.786\\
	-0.714\\
	-0.603\\
\end{array}
\right],\quad
\left[
\begin{array}{c}
	\widehat{\Theta}(1)\\
	\widehat{\Theta}(2)\\
	\widehat{\Theta}(3)\\
	\widehat{\Theta}(4)\\
\end{array}
\right]
=\left[
\begin{array}{c}
	-1.322\\
	-1.494\\
	-1.498\\
	-1.125\\
\end{array}
\right].
$$
Substituting the above optimal $u^{*}$ into the state equation \eqref{state equation},
we obtain the optimal state process:
\begin{equation}\label{optimal state}
\left\{
\begin{aligned}
\dd X^*(t)=\;&\left[(A+B\Theta)X^*+(\widehat{A}+B\widehat{\Theta})\widehat{X}^*\right]\dt
+\left[(C+D\Theta)X^*+(\widehat{C}+D\widehat{\Theta})\widehat{X}^*\right]\dw,\\
X^*(0)=\;&x,\quad \alpha(0)=i.
\end{aligned}
\right.
\end{equation}
By Proposition \ref{prop:SDE-wellpose}, we can verify that \eqref{optimal state}
indeed admits a unique solution $X^*\in L_{\mathcal{F}}^{2,r}(0,\infty;\mathbb{R})$.

\begin{figure}[h]
\centering
\includegraphics[width=4in]{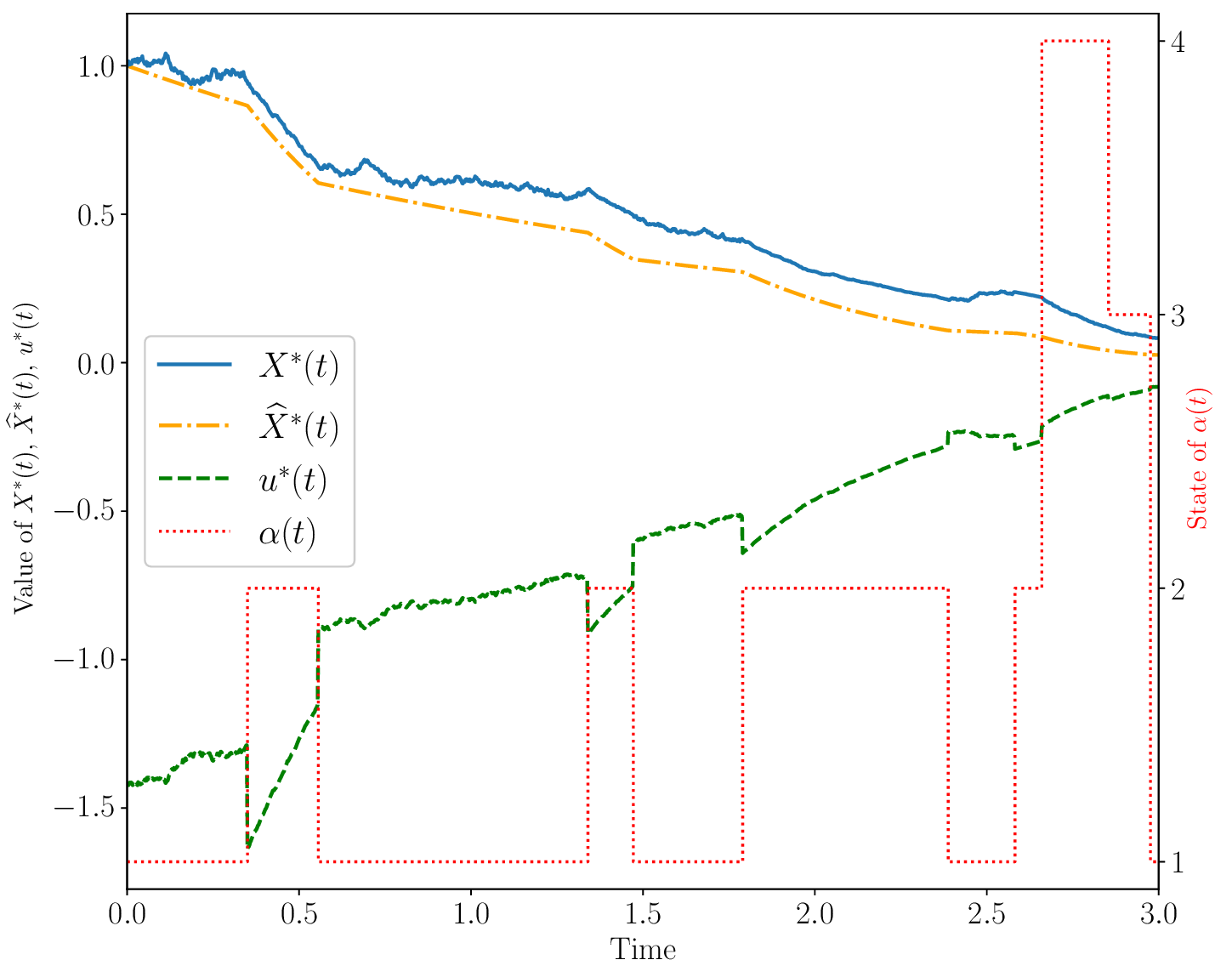}
\caption{Simulations of $X^*$, $\widehat{X}^*$, $u^*$, $\alpha$.}
\label{Fig1}
\end{figure}

In Figure \ref{Fig1}, we provide simulation results for $X^*$, $\widehat{X}^*$, $u^*$, $\alpha$. Note that:
(\romannumeral1) the trajectory of the conditional mean $\widehat{X}^*$ is more smooth than that of $X^*$,
(\romannumeral2) $X^{*}$ (and $\widehat{X}^*$) tend to 0 under $u^*$ in order to achieve
the stability of the state equation and the optimality of the cost functional, and
(\romannumeral3) $u^{*}$ (as a \emph{direct} feedback of $\alpha$)
displays a \emph{sharp} change immediately the Markov chain jumps from one state to another.

\section{Concluding remarks}\label{Section Concluding remarks}

In this paper, we systematically investigate an infinite horizon LQ optimal control problem
for mean-field switching diffusions. To ensure system stability, a discounted framework
is employed. The main contributions of this work include:
\begin{itemize}
\item Establishing the well-posedness and asymptotic properties of solutions to some infinite horizon
mean-field SDEs and backward SDEs with Markov chains;
\item Proving the solvability of two associated algebraic Riccati equations; and
\item Verifying the optimality of the proposed control strategy using a multi-step completion of squares method.
\end{itemize}

Several open questions warrant further exploration. First, while this paper assumes the standard condition
(see Assumption \eqref{H2}) on the coefficients of the LQ problem, it would be natural to extend the analysis
to the indefinite case, where the state and control weight matrices in the cost functional are not necessarily
positive definite. Second, the current formulation is set in a Markovian framework. A promising direction
is to generalize the results to a non-Markovian setting with random coefficients, as explored in
Hu et al. \cite{HSX2022AAP} and Wen et al. \cite{WLXZ2023}. Finally, in scenarios where the Markov chain has
a large state space or exhibits a two-time-scale structure, solving the Hamiltonian system and AREs -- either
analytically or numerically -- becomes computationally challenging. In such cases, a singular perturbation
approach could be employed to reduce system dimensionality and alleviate computational burden, as discussed
in Yin and Zhang \cite{YinZhang2013}.

\section*{Acknowledgements}

The authors sincerely thank Prof. Xin Zhang for helpful discussions and suggestions,
which led to much improvement of this paper.

\end{document}